\documentclass{amsart}

\usepackage{amssymb}
\usepackage[all]{xy}
\usepackage{amsmath, amsfonts, amsthm , amssymb} 
\usepackage{mathrsfs}
\usepackage{color}
\usepackage{hyperref}
\usepackage{enumitem}
\usepackage{array}

\input xypic 
\usepackage{tikz}
\usepackage{tikz-cd}
\usepackage{tikz-3dplot}

\usetikzlibrary{shapes.geometric, calc}

\usetikzlibrary{decorations.markings,shapes.arrows, shapes.symbols}
    
\tikzset{paint/.style={ draw=#1!50!black, fill=#1!50 },     decorate with/.style={decorate,decoration={shape backgrounds,shape=#1,shape size=1mm}}}
 
\newtheorem{theorem}{Theorem}[section]
\newtheorem{proposition}[theorem]{Proposition}
\newtheorem{lemma}[theorem]{Lemma}

\newtheorem{problem}[theorem]{Problem}
\newtheorem{example}[theorem]{Example}

\theoremstyle{definition}
\newtheorem{definition}[theorem]{Definition}
\newtheorem{remark}[theorem]{Remark}
\newtheorem*{acknowledgements}{Acknowledgements}

\newcounter{bean}

\newcommand{\N}{\mathbb{N}}

\newcommand{\R}{\mathbb{R}}
\newcommand{\C}{\mathbb{C}}
\newcommand{\Z}{\mathbb{Z}}
\newcommand{\tor}{g} 

\newcolumntype{C}[1]{>{\centering\let\newline\\\arraybackslash\hspace{0pt}}m{#1}}

\begin{document}

\title[]
{On the homotopy types of $4$-dimensional toric orbifolds}
\author[T. Cutler]{Tyrone Cutler}
\address{Institute of Mathematical Science, Pusan National University, Busan 46241, Republic of Korea}
\email{tcutler@math.uni-bielefeld.de}

\author[T. So]{Tseleung So}
\address{Institute of Mathematical Science, Pusan National University, Busan 46241, Republic of Korea}
\email{larry.so.tl@gmail.com}


\thanks{Both authors are supported by the National Research Foundation of Korea (NRF) grant funded by the Korea government (MSIT) (RS-2025-00555914).}

\begin{abstract}
The cohomological rigidity problem for toric orbifolds asks when an integral cohomology isomorphism implies a homotopy equivalence. In this paper we reformulate the cohomological rigidity problem in the context of $4$-dimensional toric orbifolds by introducing what we call proper isomorphisms, a variant of a concept studied by J.H.C. Whitehead. We prove that each proper isomorphism class of $4$-dimensional toric orbifolds contains at most two distinct homotopy types, and that the two classifications agree in certain special circumstances.
\end{abstract}

\maketitle

\section{Introduction}

A \emph{toric orbifold} is a $2d$-dimensional compact orbifold equipped with an effective action of a $d$-dimensional torus such that the action is locally equivariantly modelled on the standard $T^d$-action and its orbit space is a $d$-dimensional simple convex polytope. Alternatively, there is an equivalent combinatorial way to describe a toric orbifold which is recalled in Section 5. 

We are interested in the \emph{homotopy cohomological rigidity problem}: are two toric orbifolds homotopy equivalent if they have isomorphic cohomology rings? This problem is false in general, as two non-homotopy equivalent 8-dimensional toric manifolds with isomorphic cohomology rings have recently been constructed~\cite{GL}.
However, it is known that cohomological rigidity still holds for several special cases: weighted projective spaces~\cite{BFNR}, 4-dimensional quasitoric manifolds~\cite{OR}, Bott manifolds~\cite{CHJ}, certain generalized Bott manifolds~\cite{CMS-tr}, quasitoric manifolds with second Betti number equal to $2$~\cite{CPS}, and 6-dimensional quasitoric manifolds associated with Pogorelov polytopes~\cite{BEMPP}. In particular, Fu, Song, and the second author proved that $4$-dimensional toric orbifolds are homotopy equivalent if their cohomology rings are isomorphic and contain no $2$-torsion~\cite{FSS0}.

The goal of this paper is to study the cohomological rigidity problem for $4$-dimensional toric orbifolds in the general case. Previously, the second author showed that $4$-dimensional toric orbifolds with isomorphic cohomology rings have the same stable homotopy types~\cite{So}. However, improving this result to obtain an ordinary homotopy equivalence appears to be difficult. Instead, we consider a variation of the cohomological rigidity problem, formulated as follows.

Let $X$ be a $4$-dimensional toric orbifold. It is known~\cite{Fis,Jor,KMZ} that its integral cohomology groups are:
\begin{equation}\label{table_hmlgy toric orb}
\begin{tabular}{C{2.4cm}|C{1cm}|C{1cm}|C{1cm}|C{1cm}|C{1cm}|C{1cm}}
$i$		&$0$	&$1$	&$2$	&$3$	&$4$	&$\geq5$\\
\hline
$H^i(X;\Z)$	&$\Z$	&$0$	&$\Z^n$	&$\Z/\tor$	&$\Z$	&$0$
\end{tabular}
\end{equation}
where $n$ and $g$ are positive integers. If $X'$ is a second $4$-dimensional toric orbifold and $H^*(X;\Z)\cong H^*(X';\Z)$ as modules, then we say that $X$ and $X'$ are \emph{properly isomorphic} whenever there exist ring isomorphisms
\[
\varphi_t\colon H^*(X;\Z/t)\to H^*(X';\Z/t)
\quad
\text{for}\;t\in\{\tor,2\tor,\infty\},
\]
where $\Z/\infty$ means $\Z$, fitting into a commutative diagram
\[
\xymatrix{
H^*(X;\Z)\ar[r]^-{\rho}\ar[d]^-{\varphi_{\infty}}	&H^*(X;\Z/2\tor)\ar[r]^-{\rho}\ar[d]^-{\varphi_{2\tor}}	&H^*(X;\Z/\tor)\ar[d]^-{\varphi_{\tor}}\\
H^*(X';\Z)\ar[r]^-{\rho}						&H^*(X';\Z/2\tor)\ar[r]^-{\rho}							&H^*(X';\Z/\tor)
}
\]
where the $\rho$'s are reduction maps. This is a generalisation of a concept introduced by J.H.C. Whitehead~\cite{Wh}, who used the stronger version to give algebraic criteria to detect homotopy equivalences (see Remark~\ref{propisocompar} for a comparison between the two definitions).

Observe that if $X$ and $X'$ are homotopy equivalent, then they are properly isomorphic. In this paper we investigate conditions  under which the converse holds.
\begin{problem}\label{problem_new cohomological rigidity problem}
If $X,X'$ are $4$-dimensional toric orbifolds which are properly isomorphic, are they homotopy equivalent? 
\end{problem}
\noindent The main result of this paper provides a partial answer to this problem. In particular, we show that there is a single obstruction to a proper isomorphism being induced by a homotopy equivalence, and this measures the failure to commute with a certain Pontryagin square.

For a positive integer $a$, if $a=2^rq$ where $r\geq0$ and $q$ is odd, then the \emph{$2$-component} of $a$ is denoted $\nu_2(a)=2^r$.

\begin{theorem}\label{thm_main}
Let $X, X', X''$ be $4$-dimensional toric orbifolds such that their cohomology groups are isomorphic and given by Table~\eqref{table_hmlgy toric orb}. Suppose that $\tor$ is even.
\begin{enumerate}[label=(\alph*)]
\item\label{thm_main_part b}
If $X,X'$ and $X''$ are properly isomorphic, then at least two of them are homotopy equivalent.

\item\label{thm_main_part a}
Further, assume that $\nu_2(\tor)=2$ and $n=1$. Then
$X$ and $X'$ are homotopy equivalent if and only if they are properly isomorphic.
\end{enumerate}
\end{theorem}

The restriction $n=1$ in Part~\ref{thm_main_part a} of Theorem~\ref{thm_main} can be dropped, so that $\nu_2(\tor)=2$ is enough to imply that $X,X'$ are homotopy equivalent if and only if they are properly isomorphic. The details of this stronger claim will be addressed in~\cite{FSS2} (see Remark~\ref{remark_extend main proof}).

\begin{acknowledgements}
    The authors would like to thank the anonymous referee for a very thorough reading of the paper and for several suggestions that lead to its improvement.
\end{acknowledgements}


\section{Pontryagin squares and Postnikov squares}

Throughout the paper we adopt the following notation. For $t\in\N\cup\{\infty\}$, let $\Z/t$ denote the cyclic group of order $t$. 
In particular, $\mathbb{Z}/1$ is the trivial group and $\mathbb{Z}/\infty=\mathbb{Z}$.
As we will primarily consider cohomology with coefficients in $\Z/2^r$, we let
\begin{itemize}
\item
$\imath^t_s\colon H^*(-;\Z/2^t)\to H^*(-;\Z/2^s)$ for $t\leq s$ be the homomorphism induced by inclusion $\mathbb{Z}/2^t\subseteq\mathbb{Z}/2^s$,
\item
$\rho^t_s\colon H^*(-;\Z/2^t)\to H^*(-;\Z/2^s)$ for $t\geq s$ or $t=\infty$ be the reduction homomorphism,
\item
$\beta^s_t\colon H^*(-;\Z/2^s)\to H^{*+1}(-;\Z/2^t)$ for $s\in\mathbb{N}$ and $t\in\mathbb{N}\cup\infty$ be the Bockstein homomorphism associated with the exact sequence $0\to\Z/2^t\to\Z/2^{t+s}\to\Z/2^s\to0$,
\item
$\beta_t\colon H^*(-;\Z/2^t)\to H^{*+1}(-;\Z/2^t)$ be the Bockstein homomorphism $\beta^t_t$.
\end{itemize}

For an integer $\tor\geq1$, the \emph{Moore space} $P^3(\tor)$ is the mapping cone of the degree $\tor$ map $S^2\to S^2$. By Whitehead's certain exact sequence~\cite[Theorem 2.1.22]{Baues:1996} we know that
\[
\pi_3(P^3(\tor))\cong\begin{cases}
\Z/\tor		&\text{for $\tor$ odd}\\
\Z/2\tor	&\text{for $\tor$ even},
\end{cases}
\]
and is generated by the composite of the Hopf map $\eta\colon S^3\to S^2$ and the inclusion $\imath\colon S^2\hookrightarrow P^3(\tor)$. For convenience we write $\bar{\eta}=\imath\circ\eta$.

Elements of $\pi_3(P^3(\tor))$ can be detected by the \emph{Pontryagin squares}, which are cohomology operations
\[
\mathcal{P}_r\colon H^{2}(-;\mathbb{Z}/2^r)\rightarrow H^{4}(-;\mathbb{Z}/2^{r+1}),\qquad r\geq1
\]
introduced by Pontryagin in 1942~\cite{Pontrjagin:1942} and later given a careful definition by Whitehead~\cite{Wh}, who used them to classify simply connected, $4$-dimensional polyhedra up to homotopy. Analogous squaring operations for odd-primary coefficients were introduced by Whitehead in a subsequent paper~\cite{Wh1}.

The Pontryagin squares enjoy the following salient properties, proofs for which we refer the reader to~\cite{Wh,Wh1}, and ~\cite{BrowderThomas:1962}, where the operations were characterised axiomatically. 

\begin{lemma}\label{lemma_Pontryagin properties}
Let $X$ be a space and let $r$ be a positive integer. Then
\begin{enumerate}[label=(\alph*)]
\item $\rho_{r}^{r+1}(\mathcal{P}_r(x))=x^2$
\item\label{lemma_P properties reduct n sq} $\mathcal{P}_r(\rho_{r}^{r+1}(z))=z^2$
\item\label{lemma_P properties almost linear} $\mathcal{P}_r(x+y)=\mathcal{P}_r(x)+\mathcal{P}_r(y)+\imath^{r}_{r+1}(x\cup y)$
\item\label{lemma_P properties const square}
$\mathcal{P}_r(a x)=a^2\mathcal{P}_r(x)$ for any $a\in\mathbb{Z}/2^r$
\item\label{lemma_P properties refines square}
$2\mathcal{P}_r(x)=\imath^{r}_{r+1}(x^2)$,
\end{enumerate}
where $x,y\in H^2(X;\Z/2^r)$ and $z\in H^2(X;\Z/2^{r+1})$.
\end{lemma}

The following lemma, which will be used in a later section, describes how the Pontryagin squares and the modular Hopf invariant detect elements in $\pi_3(P^3(\tor))$.

\begin{lemma}\label{lemma_mapping cone c gamma}
Let $\tor$ be a positive even integer with $\nu_2(\tor)=2^r$.
For an integer $c\geq1$, let $C_{c\bar{\eta}}$ be the mapping cone of $c\bar{\eta}\colon S^3\to P^3(\tor)$. If $w\in H^2(C_{c\bar{\eta}};\Z/\tor)\cong\Z/\tor$ and $v\in H^4(C_{c\bar{\eta}};\Z)\cong\Z$ are generators, then there exists a unit $\epsilon\in\Z/2\tor$ such that
\[
w\cup w\equiv\epsilon cv\pmod{\tor}
\quad\text{and}\quad
\mathcal{P}_r(w)\equiv\epsilon cv\pmod{2^{r+1}}.
\]
\end{lemma}

\begin{proof}
Consider the homotopy commutative diagram
\[
\xymatrix{
S^3\ar[r]^-{c\eta}\ar[d]^-{=}	&S^2\ar[r]\ar[d]^-{\imath}		&C_{c\eta}\ar[d]^-{\tilde{\imath}}\\
S^3\ar[r]^-{c\bar{\eta}}			&P^3(\tor)\ar[r]\ar[d]^-{\text{pinch}}	&C_{c\bar{\eta}}\ar[d]\\
								&S^3\ar[r]^-{=}						&S^3
}
\]
where $C_{c\eta}$ is the mapping cone of $c\eta$ and $\tilde{\imath}$ is an induced map.
It is easy to see that there exist generators $x\in H^2(C_{c\eta};\Z)$ and $y\in H^4(C_{c\eta};\Z)$ such that
\[
x\cup x=cy.
\]
Since $\tilde{\imath}^*\colon H^2(C_{c\bar{\eta}};\Z/\tor)\to H^2(C_{c\eta};\Z/\tor)$ and $\tilde{\imath}^*\colon H^4(C_{c\bar{\eta}};\Z)\to H^4(C_{c\eta};\Z)$ are isomorphisms,
\[
\tilde{\imath}^*(w)\equiv\epsilon'x\pmod{\tor}
\quad\text{and}\quad
\tilde{\imath}^*(v)=\epsilon''y,
\]
where $\epsilon''\in\{1,-1\}$ and $\epsilon'$ is an integer coprime to $\tor$.
Then $\epsilon=\epsilon''(\epsilon')^2$ is a unit mod $2g$ and the
naturality of cup products gives
\[
\tilde{\imath^*}(w\cup w)\equiv (\epsilon')^2x\cup x\equiv (\epsilon')^2cy\equiv\epsilon c\tilde{\imath}^*(v)\pmod{\tor},
\]
implying that $w\cup w\equiv\epsilon cv\pmod{\tor}$.
Similarly, the naturality of the Pontryagin squares together with Lemma~\ref{lemma_Pontryagin properties}~\ref{lemma_P properties reduct n sq} gives
\[
\tilde{\imath}^*\circ\mathcal{P}_r(w)\equiv\mathcal{P}_r(\epsilon'x)\equiv(\epsilon')^2x\cup x\equiv(\epsilon')^2cy\equiv\epsilon c\tilde{\imath}^*(v)\pmod{2^{r+1}},
\]
implying that $\mathcal{P}_r(w)\equiv\epsilon cv\pmod{2^{r+1}}$.
\end{proof}

The Pontryagin squares are related to another family of cohomology operations, namely the \emph{Postnikov squares}~\cite{Postnikov:1949,Whitehead:1951}, which are the cohomology operations
\[
\mathcal{P}'_r\colon H^1(X;\mathbb{Z}/2^r)\rightarrow H^3(X;\mathbb{Z}/2^{r+1}),\qquad x\mapsto2^{r-1}\iota_{r+1}^{r}(x\cup\beta_{r}x).
\]
In~\cite[Theorem 6.1]{Thomas:1959} the Postnikov squares are shown to enjoy the following properties.
\begin{lemma}\label{lemma_Postnikov properties}
Let $X$ be a space and let $r$ be a positive integer. The following statements hold.
\begin{enumerate}[label=(\alph*)]
\item\label{lemma_Postnikov properties_homomorphism} $\mathcal{P}'_r\colon H^1(X;\mathbb{Z}/2^r)\rightarrow H^3(X;\mathbb{Z}/2^{r+1})$ is a homomorphism.
\item\label{lemma_Postnikov properties_order 2} $2\cdot\mathcal{P}'_r=0$.
\item\label{lemma_Postnikov properties_commute with suspension}
There is a commutative diagram
\[
\xymatrix{
H^1(X;\Z/2^r)\ar[r]^-{\sigma}_-\cong\ar[d]^-{\mathcal{P}'_r}&H^2(\Sigma X;\Z/2^r)\ar[d]^-{\mathcal{P}_r}\\
H^3(X;\Z/2^{r+1})\ar[r]^-{\sigma}_-\cong					&H^4(\Sigma X;\Z/2^{r+1})
}
\]
where $\sigma$ is the suspension isomorphism.
\item\label{lemma_Postnikov properties_commute with connecting maps}
Given a homotopy cofibration sequence $X\rightarrow Y\rightarrow Z$, the following square commutes
\begin{equation}\label{dgrm_Postnikov Pontryagin commute}
\begin{gathered}
\xymatrix{
H^1(X;\Z/2^r)\ar[r]^-{\delta}\ar[d]^-{\mathcal{P}'_r}	&H^2(Z;\Z/2^r)\ar[d]^-{\mathcal{P}_r}\\
H^3(X;\Z/2^{r+1})\ar[r]^-{\delta}						&H^4(Z;\Z/2^{r+1})
}
\end{gathered}
\end{equation}
where $\delta$ is the connecting homomorphism in the long exact cohomology sequence.
\end{enumerate}
\end{lemma}

\begin{proof}
The proof in~\cite{Thomas:1959}, although complete, is complicated somewhat by the level of generality it attempts. Moreover, while some details are given in~\cite{Whitehead:1951}, the important Statement~\ref{lemma_Postnikov properties_commute with suspension} is not spelled out there explicitly. On the other hand, the discussion of Statement~\ref{lemma_Postnikov properties_commute with suspension} in~\cite[\S4]{BrowderThomas:1962} seems to contain some errors. Thus we will provide full details. 

Firstly, notice that~\ref{lemma_Postnikov properties_homomorphism} follows from~\ref{lemma_Postnikov properties_commute with suspension}, as does~\ref{lemma_Postnikov properties_order 2}, since parts~\ref{lemma_P properties almost linear},~\ref{lemma_P properties const square} of Lemma~\ref{lemma_Pontryagin properties} imply that $2\cdot\mathcal{P}_r(\sigma x)=4\cdot\mathcal{P}_r(\sigma x)$, hence that $2\cdot\mathcal{P}_r(\sigma x)=0$. Furthermore, Statements~\ref{lemma_Postnikov properties_commute with suspension} and~\ref{lemma_Postnikov properties_commute with connecting maps} are actually equivalent. For applying~\ref{lemma_Postnikov properties_commute with connecting maps} to the homotopy cofibration sequence $X\rightarrow\ast\rightarrow\Sigma X$ clearly returns~\ref{lemma_Postnikov properties_commute with suspension}. Conversely, the connecting homomorphism $\delta$ in Diagram~\eqref{dgrm_Postnikov Pontryagin commute} agrees up to natural isomorphism with the composite $H^*(X;\mathbb{Z}/2^r)\xrightarrow{\sigma}H^{*+1}(\Sigma X;\mathbb{Z}/2^r)\rightarrow H^{*+1}(Z;\mathbb{Z}/2^r)$, where the second arrow is induced by a map of spaces $Z\rightarrow\Sigma X$ extending the homotopy cofibration sequence. Evidently, if~\ref{lemma_Postnikov properties_commute with suspension} holds, then the naturality of the Pontryagin squares implies that so does~\ref{lemma_Postnikov properties_commute with connecting maps}.

Thus it remains to prove~\ref{lemma_Postnikov properties_commute with suspension}, and for this it suffices to study the universal example. Thus let $K_i=K(\mathbb{Z}/2^r,i)$ for $i=1,2$ be the Eilenberg-Mac Lane space of type $(\mathbb{Z}/2^r,i)$ and let $u_i\in H^i(K_i;\mathbb{Z}/2^r)$ be its fundamental class. Then $\sigma(u_1)\in H^2(\Sigma K_1;\mathbb{Z}/2^r)$ is represented by a map $E\colon \Sigma K_1\rightarrow K_2$, that is, $E^*u_2=\sigma(u_1)$. We will evaluate $\mathcal{P}_r(\sigma u_1)=E^*\mathcal{P}_r(u_2)$.

Let $\phi\colon\Sigma K_1\wedge K_1\rightarrow\Sigma K_2$ be the Hopf construction on the loop multiplication of $K_1$ and let $P$ be the homotopy cofibre of $\phi$. Defined as such, $P$ is the projective plane of $K_2$ and there is a homotopy cofibre sequence
\begin{equation}\label{projplancofseq}
\Sigma K_1\wedge K_1\xrightarrow{\phi}\Sigma K_1\rightarrow P.
\end{equation}
 Since $\Sigma K_1\wedge K_1$ is $2$-connected, $E\circ\phi\simeq\ast$. Thus there is an induced map $\xi\colon P\rightarrow K_2$ extending $E$. The map $\xi$ is unique up to homotopy and by \cite[Proposition 3.5.1.(e)]{Harper:2002} it is $5$-connected.

Now let $H^*(-)$ denote cohomology with mod $2^{r+1}$ coefficients and consider the long exact cohomology sequence of~\eqref{projplancofseq}. Using $\xi$ to identify the low-dimensional cohomologies of $K_2$ and $P$ this becomes the exact sequence
\begin{equation}\label{extseqprojplan}
\cdots\rightarrow H^3(K_2)\xrightarrow{E^*}H^3(\Sigma K_1)\rightarrow H^1(K_1)\otimes H^1(K_1)\xrightarrow{\lambda}H^4(K_2)\xrightarrow{E^*}H^4(\Sigma K_1)\rightarrow H^4(\Sigma K_1\wedge K_1)
\end{equation}
where we identify $H^3(\Sigma K_1\wedge K_1)\cong H^1(K_1)\otimes H^1(K_1)$ via the Künneth Theorem and $\lambda$ is the composite of this identification and the connecting map of the sequence. The map $E$ is $3$-connected, so the first $E^*$ is injective. But $H_3(K_2;\mathbb{Z})$ and $H_2(K_1;\mathbb{Z})$ vanish by the Hurewicz Theorem, so the Universal Coefficient Theorem shows that both $H^3(K_2)$ and $H^3(\Sigma K_1)$ are isomorphic to $\mathbb{Z}/2^r$, and this implies that this $E^*$ must be isomorphic. Thus $\lambda$ is injective.

Now, $v_i=\imath_{r+1}^{r}(u_i)$ generates $H^i(K_i)$ for $i=1,2$, and $v_1\otimes v_1$ generates $H^1(K_1)\otimes H^1(K_1)\cong\Z/2^r$. By~\cite[Theorem 2.4]{Thomas:1961}, $\lambda$ satisfies $\lambda(v_1\otimes v_1)=v_2^2$. Since $2\cdot\mathcal{P}_r(u_2)=v_2^2$ by Lemma~\ref{lemma_Pontryagin properties}~\ref{lemma_P properties refines square}, we obtain
\[
\lambda(v_1\otimes v_1)=2\cdot\mathcal{P}_r(u_2).
\]
From the exactness of~\eqref{extseqprojplan} it follows that $E^*\mathcal{P}_r(u_2)$ is nontrivial and has order $2$ in $H^4(\Sigma K_1)\cong H^3(K_1)$. But $H^3(K_1;\mathbb{Z}/2^r)\cong\mathbb{Z}/2^r$ is generated by $u_1\cup\beta_r u_1$ by~\cite[Example 3E.2]{Hat} and its easy to see that $\imath_{r+1}^{r}\colon H^3(K_1;\mathbb{Z}/2^r)\rightarrow H^3(K_1;\mathbb{Z}/2^{r+1})=H^3(K_1)$ is isomorphic. Therefore $\mathcal{P}'_r(u_1)=2^{r-1}\cdot\imath_{r+1}^{r}(u_1\cup\beta_r u_1)$ is the unique element of order $2$ in $H^3(K_1)$. Thus
\[
\mathcal{P}_r(\sigma u_1)=E^*\mathcal{P}_r(u_2)=\sigma\mathcal{P}'_r(u_1),
\]
which is what was needed to be shown.
\end{proof}

We close this section by discussing the action of the Postnikov operations on three dimensional lens space. Given coprime integers $a$ and $b$, the lens space $L(b,a)$ is defined as
\[
L(b,a)=S^3/[(z_1,z_2)\sim(\omega z_1,\omega^{a} z_2)]
\]
where $S^3$ is regarded as the unit sphere in $\R^4\cong\C^2$ and $\omega=\exp(\frac{2\pi\sqrt{-1}}{b})$ is a $b$\textsuperscript{th} root of unity. We denote $L(b,a)$ by $L$ if there is no confusion.
Suppose $\nu_2(b)=2^r$.
Then for any positive integer $s$ the Universal Coefficient Theorem implies that
\[
H^i(L;\Z/2^s)=\begin{cases}
\Z/2^{\min\{r,s\}}	&i=1,2\\
\Z/2^s				&i=0,3\\
0					&\text{else}.
\end{cases}
\]

\begin{proposition}\label{lemma_Postnikov sq in L}
Let $L=L(b,a)$ be a lens space such that $\nu_2(b)=2^r$ and $r\geq1$.
Then the Postnikov square $\mathcal{P}'_s\colon H^1(L;\Z/2^s)\to H^3(L;\Z/2^{s+1})$ is given by
\[
\mathcal{P}'_s(x_s)\equiv\begin{cases}
2^r\pmod{2^{r+1}}		&\text{if }r=s\\
0						&\text{otherwise}
\end{cases}
\]
where $x_s\in H^1(L;\Z/2^s)$ is a generator.
\end{proposition}
\begin{proof}
First, consider the case where $s> r$.
Since $x_s$ has order $2^r\leq 2^{s-1}$, the cup product $x_s\cup\beta_sx_s$ has order at most $2^{s-1}$. Hence by definition
\[
\mathcal{P}'_s(x_s)=2^{s-1}\iota_{s+1}^{s}(x_s\cup\beta_{s} x_s)=0.
\]

Next, consider the case where $s\leq r$. According to~\cite[Example 3.41]{Hat}
\[
y_r=\beta_{r}(x_r)\in H^2(L;\Z/2^r)
\quad\text{and}\quad
z_r=x_r\cup y_r\in H^3(L;\Z/2^r)
\]
are generators.
Let $x_s\in H^1(L;\Z/2^s)$, $y_s\in H^2(L;\Z/2^s)$, and $z_s\in H^3(L;\Z/2^s)$ be the mod-$2^s$ reductions of $x_r,y_r$ and $z_r$, respectively.
Since $\rho_{s}^{r}\colon H^*(L;\Z/2^r)\to H^*(L;\Z/2^s)$ is surjective, $x_s,y_s,z_s$ generate their respective groups and satisfy
\[
x_s\cup y_s=z_s.
\]
Moreover, $\iota_{s+1}^{s}(z_s)=2z_{s+1}$ for $s\leq r$, where $z_{r+1}$ is a generator of $H^3(L;\Z/2^{r+1})$.

Consider the commutative diagram
\[
\xymatrix{
H^1(L;\Z/2^s)\ar[r]^-{\beta_\infty^{s}}\ar[dr]_-{\beta_{s}}	&H^2(L;\Z)\ar[d]^-{\rho^{\infty}_{s}}\\
														&H^2(L;\Z/2^s).
}
\]
As $\beta_\infty^{s}(x_s)=2^{r-s}y$ for some generator $y\in H^2(L;\Z)$, we have
\[
\beta_{s}(x_s)=\rho_{s}^\infty\circ\beta_\infty^{s}(x_s)=2^{r-s}\rho_{s}^\infty(y)=2^{r-s}t y_s
\]
for some unit $t\in\mathbb{Z}/2^s$. This gives
\[
\mathcal{P}'_s(x_s)=2^{s-1}\iota_{s+1}^{s}(x_s\cup\beta_{s}x_s)=2^{r-1}t\iota_{s+1}^{s}(x_s\cup y_s)=2^rt z_{s+1}.
\]
If $s<r$, then $z_{s+1}$ has order at most $2^r$, so this expression vanishes. If $s=r$, then $2^rt z_{r+1}=2^r z_{r+1}$ is the unique element of order $2$ in $H^3(L;\mathbb{Z}/2^{r+1})$.
\end{proof}


\section{Cellular bases for 4-dimensional CW-complexes}

In this section we recall the homotopy theoretical tools used for studying the cohomological rigidity of $4$-dimensional toric orbifolds in~\cite{FSS0}.

For integers $n\geq 0$ and $\tor\geq 1$, let $\mathscr{C}_{n,\tor}$ be the full subcategory of $Top$ consisting of all $4$-dimensional simply-connected CW-complexes of the form
\[
X=\left(\bigvee^n_{i=1}S^2_i\vee P^3(\tor)\right)\cup_f e^4
\]
where each $S^2_i$ is a pointed 2-sphere,
\[
f\colon S^3=\partial e^4\to\bigvee^n_{i=1}S^2_i\vee P^3(\tor)
\]
is the attaching map of the top cell, and the subscripts $i$ are introduced solely for bookkeeping purposes. Let $[S^2_i],[P^3(\tor)]\in H_2(X;\Z)$ and $[e^4]\in H_4(X;\Z)$ be homology classes representing $S^2_i$, the bottom cell of $P^3(\tor)$, and the top cell $e^4$, respectively.
Then the homology of $X$ is given by
\[
H_*(X;\Z)\cong\Z\langle1, [S^2_1],\ldots,[S^2_n],[e^4]\rangle\oplus\Z/\tor\langle[P^3(\tor)]\rangle.
\]

\begin{definition}\label{dfn_basis}
Let $X$ be a CW-complex in $\mathscr{C}_{n,\tor}$.
 A \emph{cellular basis} for $X$ is a collection
\[
\{u_1,\ldots,u_n;w;v\}
\]
of cohomology classes of $X$, where
\begin{itemize}
\item
$u_i\in H^2(X;\Z)$ and $v\in H^4(X;\Z)$ are dual to $[S^2_i]$ and $[e^4]$
\item
$w\in H^2(X;\Z/\tor)$ is dual to the mod-$\tor$ reduction of $[P^3(\tor)]$.
\end{itemize}
\end{definition}

Given $X\in\mathscr{C}_{n,\tor}$, a cellular basis $\{u_1,\ldots,u_n;w;v\}$ for $X$ determines $H^*(X;\Z)$ and $H^*(X;\Z/\tor)$ additively. In particular,
\begin{itemize}
\item
$\{u_1,\ldots,u_n\}$ forms a basis for $\tilde{H}^2(X;\Z)$;
\item
$\{u_1,\ldots,u_n,w\}$ forms a basis for $\tilde{H}^2(X;\Z/\tor)$, where the $u_i$'s here are reduced mod-$\tor$;
\item
$\beta(w)\in H^3(X;\Z/\tor)$ is a generator, where $\beta\colon H^2(X;\Z/\tor)\to H^3(X;\Z/\tor)$ is the Bockstein homomorphism associated to the coefficient sequence $\mathbb{Z}/g\rightarrow\mathbb{Z}/g^2\rightarrow\mathbb{Z}/g$;
\item
$v$ and its mod-$\tor$ reduction are generators of $H^4(X;\Z)$ and $H^4(X;\Z/\tor)$, respectively.
\end{itemize}

Assume that $\tor$ is even. The attaching map $f$ represents an element in $\pi_3(\bigvee^n_{i=1}S^2_i\vee P^3(\tor))$. By the Hilton-Milnor Theorem it can be written
\begin{equation}\label{eqn_hilton milnor}
f\simeq\sum^n_{i=1}a_{ii}\eta_i+\sum_{1\leq j<k\leq n}a_{jk}w_{jk}+\sum^n_{\ell=1}b_{\ell}\bar{w}_{\ell}+c\bar{\eta}
\end{equation}
where $a_{ii},a_{jk}\in\Z,b_{\ell}\in\Z/\tor$ and $c\in\Z/2\tor$ are coefficients, and $\eta_i,w_{jk},\bar{w}_{\ell},\bar{\eta}$ are the composites
\begin{itemize}
\item
$\eta_i\colon S^3\overset{\eta}{\longrightarrow}S^2_i\hookrightarrow X$
\item
$w_{jk}\colon S^3\overset{[1,1]}{\longrightarrow}S^2_j\vee S^2_k\hookrightarrow X$
\item
$\bar{w}_{\ell}\colon S^3\overset{[1,\imath]}{\longrightarrow}S^2_{\ell}\vee P^3(\tor)\hookrightarrow X$
\item
$\bar{\eta}\colon S^3\overset{\eta}{\longrightarrow}S^2\hookrightarrow P^3(\tor)\hookrightarrow X$,
\end{itemize}
where $\eta$ is the Hopf map and $[1,1]$ and $[1,\imath]$ are Whitehead products.
Note that the coefficients in Equation~\eqref{eqn_hilton milnor} depend on the cellular basis $\mathcal{B}=\{u_1,\ldots,u_n;w;v\}$, since it determines the orientations of $S^2_i$ and $S^3=\partial e^4$ as well as the inclusion $\imath$ of the bottom cell into $P^3(\tor)$.
To emphasize this dependence we express these coefficients in terms of a triple
\[
(A,\textbf{b},c)\in \text{Mat}_n(\Z)\oplus(\Z/\tor)^n\oplus\Z/2\tor,
\]
where $A=(a_{jk})_{1\leq j,k\leq n}$ is a symmetric integral $n\times n$ matrix, $\textbf{b}=(b_1,\ldots,b_n)$ is an $n$-dimensional $\Z/\tor$-vector, and $c$ is a $\Z/2\tor$-coefficient.
We call $(A,\textbf{b},c)$ \emph{the attaching map coefficients} associated with $\mathcal{B}$.\footnote{In~\cite{FSS0}, where $\tor$ is assumed to be odd, $c$ is taken to be a $\Z/\tor$-coefficient and this triple is called the \emph{cellular cup product representation matrix}. Here we adopt a different name since this triple is no longer determined by the cup products in $H^*(X)$.} 

\begin{lemma}\label{lemma_hilton milnor coeff and cup prod}
Let $\{u_1,\ldots,u_n;w;v\}$ be a cellular basis of $X\in\mathscr{C}_{n,\tor}$ and let $(A,\textbf{b},c)$ be its associated attaching map coefficients.
If $\nu_2(g)=2^r$ where $r\geq 1$, then
\begin{enumerate}[label=(\alph*)]
\item\label{lemma_hilton milnor coeff and cup prod_u cup u}
$u_j\cup u_k=a_{jk}v$
\item\label{lemma_hilton milnor coeff and cup prod_u cup w}
$u_{\ell}\cup w\equiv b_{\ell}v\pmod{\tor}$
\item\label{lemma_hilton milnor coeff and cup prod_w cup w}
$w\cup w\equiv cv\pmod{\tor}$
\item\label{lemma_hilton milnor coeff and cup prod_P(w)}
$\mathcal{P}_r(w)\equiv cv\pmod{2^{r+1}}$.
\end{enumerate}
\end{lemma}

\begin{proof}
Statements~\ref{lemma_hilton milnor coeff and cup prod_u cup u} and~\ref{lemma_hilton milnor coeff and cup prod_u cup w} were proved in~\cite[Lemma 4.1]{FSS0}.

Here we prove~\ref{lemma_hilton milnor coeff and cup prod_w cup w} and~\ref{lemma_hilton milnor coeff and cup prod_P(w)}.
Write $X_3=\bigvee^n_{i=1}S^2\vee P^3(\tor)$ for short. Let $q\colon X_3\rightarrow P^3(\tor)$ be the pinch map and let $Y$ be the cofibre of the composite $q\circ f\simeq c\bar{\eta}$. Now consider the homotopy commutative diagram of cofibre sequences
\[
\xymatrix{S^3\ar@{=}[d]\ar[r]^-f&X_3\ar[r]\ar[d]^-q&X\ar[d]^-{\tilde{q}}\ar[r]&S^4\ar@{=}[d]\\
S^3\ar[r]^-{c\bar{\eta}}&P^3(\tor)\ar[r]&Y\ar[r]&S^4}
\]
where $\tilde{q}$ is an induced map collapsing $\bigvee^n_{i=1}S^2_i$ in $X$.
Let $\tilde{w}\in H^2(Y;\Z/\tor)$ and $\tilde{v}\in H^4(Y;\Z)$ be generators such that
$\tilde{q}^*(\tilde{w})=w$ and $\tilde{q}^*(\tilde{v})=v$.
By Lemma~\ref{lemma_mapping cone c gamma} they satisfy
\[
\tilde{w}\cup \tilde{w}\equiv\epsilon c\tilde{v}\pmod{\tor}
\quad\text{and}\quad
\mathcal{P}_r(\tilde{w})\equiv\epsilon c\tilde{v}\pmod{2^{r+1}},
\]
for some unit $\epsilon\in\Z/2g$.
Note that $Y$ is a complex in $\mathscr{C}_{0,\tor}$ and $\{\tilde{w};\tilde{v}\}$ is its cellular basis.
Revising the proof of Lemma~\ref{lemma_mapping cone c gamma} and taking $x\in H^2(C_{c\eta};\Z)$ and $y\in H^4(C_{c\eta};\Z)$ there to be cellular basis generators, we see that $\epsilon$ can be taken to be $1$. It then follows from the naturality of cup products and Pontryagin squares that
\[
w\cup w\equiv cv\pmod{\tor}
\quad\text{and}\quad
\mathcal{P}_r(w)\equiv cv\pmod{2^{r+1}},
\]
so the proof is complete.
\end{proof}

\section{Proper isomorphisms}

Let $X$ and $X'$ be CW-complexes in $\mathscr{C}_{n,\tor}$.
In this section we study a cohomological condition under which they are homotopy equivalent.
Clearly, a homotopy equivalence between $X$ and $X'$ induces isomorphisms between their cohomology rings $H^*(X;R)$ and $H^*(X';R)$ for any coefficient ring $R$.
On the other hand, the attaching maps of the top cells are of the form~\eqref{eqn_hilton milnor} and are sums of Hopf maps and Whitehead products with coefficients in $\Z,\Z/2\tor$, and $\Z/\tor$. 
It is therefore natural to expect that any abstract isomorphisms $H^*(X;R)\to H^*(X';R)$ for $R=\Z,\Z/2\tor$ and $\Z/\tor$ would reflect some homotopy-theoretic information. This leads to the following definition.
\begin{definition}\label{def_proper isom}
Given $X,X'\in\mathscr{C}_{n,\tor}$, a \emph{proper isomorphism} between them is a collection
\[
\varphi=\{\varphi_t\colon H^*(X;\Z/t)\to H^*(X';\Z/t)\mid t\in\{\tor,2\tor,\infty\}\}
\]
of graded ring isomorphisms such that the diagram
\begin{equation}\label{dgrm_proper isom commute rho}
\begin{gathered}
\xymatrix{
H^*(X;\Z)\ar[r]^-{\rho}\ar[d]^-{\varphi_{\infty}}	&H^*(X;\Z/2\tor)\ar[r]^-{\rho}\ar[d]^-{\varphi_{2\tor}}	&H^*(X;\Z/\tor)\ar[d]^-{\varphi_{\tor}}\\
H^*(X';\Z)\ar[r]^-{\rho}						&H^*(X';\Z/2\tor)\ar[r]^-{\rho}							&H^*(X';\Z/\tor)
}
\end{gathered}
\end{equation}
commutes, where $\rho$'s are reductions. We denote it by $\varphi\colon X' \dashrightarrow X$ since $\varphi$ is not an actual map between spaces. We say that $X$ and $X'$ are \emph{properly isomorphic}, written $X\cong_{\text{prop}}X'$, if there exists a proper isomorphism between them.
\end{definition}

Proper isomorphisms $\varphi\colon X' \dashrightarrow X$ and $\varphi'\colon X'' \dashrightarrow X'$ may be composed to give a proper isomorphism $\varphi\circ\varphi'\colon X'' \dashrightarrow X$ defined by
\begin{equation}\label{eq_composition}
(\varphi\circ\varphi')_{t}=\varphi'_{t}\circ\varphi_{t}\colon H^*(X;\Z/t)\to H^*(X'';\Z/t).
\end{equation}
Similarly, any proper isomorphism $\varphi\colon X' \dashrightarrow X$ has an inverse $\varphi^{-1}\colon X \dashrightarrow X'$ defined by
\[
(\varphi^{-1})_{t}=(\varphi_{t})^{-1}\colon H^*(X';\Z/t)\to H^*(X;\Z/t).
\]
We leave it to the reader to check that $\varphi'\circ\varphi$ and $\varphi^{-1}$ are again proper isomorphisms. Since the identity morphisms on $H^*(X;R)$ assemble into a proper morphism $id\colon X\dashrightarrow X$, being properly isomorphic defines an equivalence relation on $\mathscr{C}_{n,\tor}$.

Being properly isomorphic is a weaker notion than homotopy equivalence. If $h\colon X'\to X$ is a homotopy equivalence, then the family $\{h^*\colon H^*(X;\Z/t)\rightarrow H^*(X';\Z/t)\mid t\in\{\tor,2\tor,\infty\}\}$ is a proper isomorphism between them. We would like to determine conditions under which a proper isomorphism arises from a homotopy equivalence.

For any $X\in\mathscr{C}_{n,\tor}$, let $PI(X)=\{X'\in\mathscr{C}_{n,\tor}\mid X\cong_{\text{prop}}X'\}$ be the collection of CW-complexes in~$\mathscr{C}_{n,\tor}$ that are properly isomorphic to $X$. We define
\[
h(X)=|PI(X)/\simeq|
\]
to count the number of distinct homotopy types of CW-complexes in $PI(X)$.

Consider the cohomological rigidity problem in the sense of Problem~\ref{problem_new cohomological rigidity problem}.
The index $h(X)$ measures how close a space $X\in\mathscr{C}_{n,\tor}$ is to being cohomologically rigid. In particular, $h(X)=1$ if and only if $X$ is cohomologically rigid.
Therefore Problem~\ref{problem_new cohomological rigidity problem} can be restated as follows.

\begin{problem}
For which complexes $X\in\mathscr{C}_{n,\tor}$ does it holds that $h(X)=1$?
\end{problem}

Let $\varphi\colon X'\dashrightarrow X$ be a proper isomorphism. A necessary condition for it to be induced by a homotopy equivalence is that it commutes with all cohomology operations. More precisely, for any cohomology operation $\psi\colon H^*(-;\Z/t)\to H^*(-;\Z/s)$, the following diagram commutes
\[
\xymatrix{
H^*(X';\Z/t)\ar[r]^-{\varphi_{t}}\ar[d]^-{\psi}	&H^*(X;\Z/t)\ar[d]^-{\psi}\\
H^*(X';\Z/s)\ar[r]^-{\varphi_{s}}					&H^*(X;\Z/s).
}
\]

We will prove that a criterion for two CW complexes $X,X'\in\mathscr{C}_{n,\tor}$ to be homotopy equivalent is the existence of a proper isomorphism between them that commutes with the Pontryagin squares. To begin with, we prove a preparatory lemma.

\begin{lemma}\label{lemma_basis define skeleton}
For $X,X'\in\mathscr{C}_{n,\tor}$, let $\varphi\colon X'\dashrightarrow X$ be a proper isomorphism. If $\{u_1,\ldots,u_n;w;v\}$ is a cellular basis for $X$, then there exists a homotopy equivalence
\[
X'\simeq\left(\bigvee^n_{i=1}S^2_i\vee P^3(\tor)\right)\cup_{f'} e^4
\]
such that $\{\varphi_{\infty}(u_1),\ldots,\varphi_{\infty}(u_n);\varphi_{\tor}(w);\varphi_{\infty}(v)\}$ is a cellular basis for $X'$ with respect to this cellular structure.
\end{lemma}

\begin{proof}
For convenience, write $u'_i=\varphi_{\infty}(u_i),w'=\varphi_{\tor}(w)$ and $v'=\varphi_{\infty}(v)$. Since $X'$ is simply connected, the Hurewicz homomorphism gives an isomorphism
\[
\mathrm{hur}\colon \pi_2(X')\overset{\cong}{\longrightarrow}H_2(X';\Z)\cong\Z^n\oplus\Z/\tor.
\]
Let $\mu_i\colon S^2\to X'$ and $\tilde{\omega}\colon S^2\to X'$  be maps such that $\mathrm{hur}(\mu_i)$ and $\mathrm{hur}(\tilde{\omega})$ are dual to $u'_i$ and $w'$, respectively. Since $\tilde{\omega}$ has order $\tor$, it extends to a map $\omega\colon P^3(\tor)\to X$.
Take $\imath$ to be the wedge sum
\[
\imath=\bigvee^n_{i=1}\mu_i\vee\omega\colon\bigvee^n_{i=1}S^2\vee P^3(\tor)\to X'.
\]
Then $\imath$ factors through the $3$-skeleton $\mathrm{skel}_3X'$ of $X'$. Since it induces an isomorphism in homology in degrees $\leq 3$, the Whitehead Theorem implies that $\imath$ is a homotopy equivalence between $\bigvee^n_{i=1}S^2\vee P^3(\tor)$ and $\mathrm{skel}_3X'$. Let $f'$ be the composite
\[
f'\colon S^3\longrightarrow\mathrm{skel}_3X'\overset{\jmath}{\longrightarrow}\bigvee^n_{i=1}S^2\vee P^3(\tor),
\]
where the first map attaches the top cell of $X'=\mathrm{skel}_4X$ and $\jmath$ is a homotopy inverse to $\imath$. Now let~$C_{f'}$ be the mapping cone of $f'$ and consider the homotopy commutative diagram
\[
\xymatrix{
S^3\ar[r]\ar[d]^-{=}	&\mathrm{skel}_3X'\ar[r]\ar[d]^-{\jmath}	&X'\ar[d]^-{\tilde{\imath}}\\
S^3\ar[r]^-{f'}		&\bigvee^n_{i=1}S^2_i\vee P^3(\tor)\ar[r]	&C_{f'}
}
\]
where $\tilde{\imath}$ is an induced map. The map $\tilde{\imath}$ is a homology isomorphism by the Five Lemma, and hence is a homotopy equivalence by Whitehead's Theorem. By construction, $\{u'_1,\ldots,u'_n;w';v'\}$ is a cellular basis for $C_{f'}\simeq X'$, so the proof is complete.
\end{proof}

Suppose that $\nu_2(g)=2^{r+1}$. We define ring homomorphisms
\[
\varphi_{2^{r+1}}\colon H^*(X;\Z/2^{r+1})\to H^*(X';\Z/2^{r+1})
\quad\text{and}\quad
\varphi_{2^{r}}\colon H^*(X;\Z/2^{r})\to H^*(X';\Z/2^{r})
\]
via the following diagrams
\[
\xymatrix{H^*(X;\Z/2g)\ar[r]^-{\varphi_{2g}}\ar[d]^-{\rho}&H^*(X';\Z/2g)\ar[d]^-{\rho}\\
H^*(X;\Z/2^{r+1})\ar@{-->}[r]^-{\varphi_{2^{r+1}}}&H^*(X';\Z/2^{r+1})}
\quad\quad
\xymatrix{H^*(X;\Z/g)\ar[r]^-{\varphi_{g}}\ar[d]^-{\rho}&H^*(X';\Z/g)\ar[d]^-{\rho}\\
H^*(X;\Z/2^r)\ar@{-->}[r]^-{\varphi_{2^r}}&H^*(X';\Z/2^r)}
\]
where $\rho$'s are the reductions. Since the reductions are surjective, $\varphi_{2^{r+1}}$ and $\varphi_{2^r}$ are well-defined and unique. Since $\varphi_{2g}$ and $\varphi_g$ are ring isomorphisms, so are $\varphi_{2^{r+1}}$ and $\varphi_{2^r}$. Moreover, using the above diagrams and Diagram~\eqref{dgrm_proper isom commute rho} one can check that
$\varphi_{2^r}\circ\rho^{r+1}_{r}=\rho^{r+1}_{r}\circ\varphi_{2^{r+1}}$.

\begin{lemma}\label{lemma_criteria}
Let $X$ and $X'$ be complexes in $\mathscr{C}_{n,\tor}$, where $\nu_2(\tor)=2^r$. Then $X\simeq X'$ if and only if
there exists a proper isomorphism $\varphi\colon X'\dashrightarrow X$ which commutes with the Pontryagin square $\mathcal{P}_r\colon H^2(-;\Z/2^r)\to H^4(-;\Z/2^{r+1})$ in the sense that
\[
\mathcal{P}_r\circ\varphi_{2^r}=\varphi_{2^{r+1}}\circ\mathcal{P}_r.
\]
\end{lemma}

\begin{proof}
It is clear that any proper isomorphism induced by a homotopy equivalence commutes with $\mathcal{P}_r$. Thus we need only prove necessity.

Thus suppose such a proper isomorphism $\varphi$ exists. Let $\mathcal{B}=\{u_1,\ldots,u_n;w;v\}$ be a cellular basis for $X$. By Lemma~\ref{lemma_basis define skeleton}, $X'$ has a cellular structure such that
\[
\mathcal{B}'=\{u'_1=\varphi_{\infty}(u_1),\ldots,u'_n=\varphi_{\infty}(u_n);w'=\varphi_{\tor}(w);v'=\varphi_{\infty}(v)\}
\]
is its cellular basis.
Let $f,f'\colon S^3\longrightarrow \bigvee^n_{i=1}S^2_i\vee P^3(\tor)$ be the attaching maps of the top cell in $X$ and $X'$,
and let $(A,\textbf{b},c)$ and $(A',\textbf{b}',c')$ be the attaching map coefficients associated with $\mathcal{B}$ and $\mathcal{B}'$, respectively.
Then we have
\[
f\simeq\sum^n_{i=1}a_{ii}\eta_i+\sum_{1\leq j<k\leq n}a_{jk}w_{jk}+\sum^n_{\ell}b_{\ell}\bar{w}_{\ell}+c\bar{\eta}
\]
and
\[
f'\simeq\sum^n_{i=1}a'_{ii}\eta_i+\sum_{1\leq j<k\leq n}a'_{jk}w_{jk}+\sum^n_{\ell}b'_{\ell}\bar{w}_{\ell}+c'\bar{\eta}
\]
for some coefficients $a_{ij},a'_{ij}\in\Z$, $b_{k},b'_k\in\Z/\tor$, and $c,c'\in\Z/2\tor$.
By Lemma~\ref{lemma_hilton milnor coeff and cup prod}, we have
\[
u_j\cup u_k=a_{jk}v
\quad\text{and}\quad
u'_j\cup u'_k=a'_{jk}v'
\]
for $1\leq j,k\leq n$.
Since $\varphi_{\infty}$ is a ring isomorphism, we get $u'_j\cup u'_k=a_{jk}v'$ and hence $a_{jk}=a'_{jk}$.
A similar argument shows that $b_{\ell}\equiv b'_{\ell}\pmod{\tor}$ for $1\leq\ell\leq n$, and $c\equiv c'\pmod{\tor}$.
Moreover, by assumption, $\varphi$ commutes with $\mathcal{P}_r$, so $\mathcal{P}_r(w)\equiv\mathcal{P}_r(w')\pmod{2^{r+1}}$ and hence $c\equiv c'\pmod{2\tor}$. It follows that $f\simeq f'$ and hence that $X\simeq X'$.
\end{proof}

\begin{remark}\label{propisocompar}
In Whitehead's paper~\cite{Wh}, a \emph{proper isomorphism} between simply-connected $4$-dimensional CW-complexes $X$ and $X'$ is defined to be a collection of ring isomorphisms
\[
\{\varphi_t\colon H^*(X;\Z/t)\to H^*(X';\Z/t)\mid 2\leq t\leq \infty\},
\]
where $\Z/t$ means $\Z$ for $t=\infty$, such that they commute with
\begin{itemize}
\item
reductions $\rho\colon H^*(-;\Z/t)\to H^*(-;\Z/s)$ for $s$ dividing $t$ or $t=\infty$,
\item
inclusions $\imath\colon H^*(-;\Z/t)\to H^*(-;\Z/s)$ for $t$ dividing $s$,
\item 
Bockstein homomorphisms $\beta\colon H^*(-;\Z/s)\to H^{*+1}(-;\Z/t)$ for $s\in\mathbb{N}$ and $t\in\mathbb{N}\cup\infty$,
\item
Pontryagin squares $\mathcal{P}_{\ell}\colon H^2(-;\Z/2^{\ell})\to H^4(-;\Z/2^{\ell+1})$ for $\ell\geq 1$.
\end{itemize}
Whitehead proved that $X$ and $X'$ are homotopy equivalent if and only if there exists such a proper isomorphism between them.
Lemma~\ref{lemma_criteria} may be regarded as a refinement of this result. Whitehead considers ring isomorphisms for all coefficients $\Z/t$ and requires them to commute with more cohomology operations. 
In contrast, Definition~\ref{def_proper isom} and Lemma~\ref{lemma_criteria} consider ring isomorphisms only for the coefficients $\Z,\Z/2\tor$, and $\Z/\tor$, and require them to commute only with the three reduction maps and the single Pontryagin square $\mathcal{P}_r$. The trade-off is that Lemma~\ref{lemma_criteria} applies only to a smaller class of spaces in $\mathscr{C}_{n,\tor}$. A proper isomorphism in the sense of Whitehead is a proper isomorphism in our sense, but Example~\ref{ex_propisonothomo} below shows that Whitehead's notion is strictly stronger than ours.
\end{remark}

Motivated by Lemma~\ref{lemma_criteria} we define the commutator $[\mathcal{P}_r,\varphi]$ of a proper isomorphism $\varphi\colon X'\dashrightarrow X$ and a Pontryagin square $\mathcal{P}_r$ to be the function
\begin{equation}\label{eq_commutator}
[\mathcal{P}_r,\varphi]=\mathcal{P}_{r}\circ\varphi_{2^r}-\varphi_{2^{r+1}}\circ\mathcal{P}_r\colon H^2(X;\Z/2^r)\to H^4(X';\Z/2^{r+1}).
\end{equation}
Then $\mathcal{P}_r$ and $\varphi$ commute if and only if $[\mathcal{P}_r,\varphi]=0$. The following lemma describes some basic properties of the commutator.

\begin{lemma}\label{lemma_properties of commutator}
Given $X,X'\in\mathscr{C}_{n,\tor}$, let $\varphi\colon X'\dashrightarrow X$ be a proper isomorphism. 
Then
\begin{enumerate}[label=(\alph*)]
\item\label{lemma_commutator property_linear}
$[\mathcal{P}_r,\varphi](x+y)=[\mathcal{P}_r,\varphi](x)+[\mathcal{P}_r,\varphi](y)$ for $x,y\in H^2(X;\Z/2^r)$ 
\item\label{lemma_commutator property_2^r}
$2\cdot[\mathcal{P}_r,\varphi]=0$
\item\label{lemma_commutator property_zero on integral}
$[\mathcal{P}_r,\varphi]\circ\rho_{r}^{r+1}=0$.
\end{enumerate}
Moreover, let $X''\in\mathscr{C}_{n,g}$ and let $\varphi'\colon X''\dashrightarrow X'$ be another proper isomorphism. Then
\begin{enumerate}[label=(\alph*)]
\setcounter{enumi}{3}
    \item\label{lemma_commutator of composition}
    $[\mathcal{P}_r,\varphi\circ\varphi']=\varphi'_{2^{r+1}}\circ[\mathcal{P}_r,\varphi]+[\mathcal{P}_r,\varphi']\circ\varphi_{2^{r}}$
    \item\label{lemma_commutator of inverse}
    $[\mathcal{P}_r,\varphi^{-1}]=\varphi^{-1}_{2^{r+1}}\circ[\mathcal{P}_r,\varphi]\circ\varphi^{-1}_{2^{r}}$.
\end{enumerate}
\end{lemma}

\begin{proof}
First, we prove~\ref{lemma_commutator property_linear}. By Lemma~\ref{lemma_Pontryagin properties}~\ref{lemma_P properties almost linear} we have
\[
[\mathcal{P}_r,\varphi](x+y)
=[\mathcal{P}_r,\varphi](x)+[\mathcal{P}_r,\varphi](y)+\imath^r_{r+1}\circ\varphi_{2^r}(x\cup y)-\varphi_{2^{r+1}}\circ\imath^r_{r+1}(x\cup y).
\]
It remains to show that $\imath^r_{r+1}\circ\varphi_{2^r}(x\cup y)-\varphi_{2^{r+1}}\circ\imath^r_{r+1}(x\cup y)$ is zero. In fact, this will follow from the stronger result that $\imath^r_{r+1}\circ\varphi_{2^r}=\varphi_{2^{r+1}}\circ\imath^r_{r+1}$ as maps $H^4(X;\mathbb{Z}/2^r)\rightarrow H^4(X';\mathbb{Z}/2^{r+1})$. To see this, notice that $\imath^r_{r+1}\circ\rho^\infty_{r}=\rho_{r+1}^\infty\circ 2$ where $2$ denotes multiplication by $2$. Then
\[
\varphi_{2^{r+1}}\circ\imath^r_{r+1}\circ\rho_{r}^\infty=\varphi_{2^{r+1}}\circ\rho_{r+1}^\infty\circ 2=\rho_{r+1}^\infty\circ\varphi_{\infty}\circ 2
\]
and similarly 
\[
\imath^r_{r+1}\circ\varphi_{2^r}\circ\rho_{r}^\infty=\imath^r_{r+1}\circ\rho_{r}^\infty\circ\varphi_{\infty}=\rho_{r+1}^\infty\circ2\circ\varphi_{\infty}=\rho_{r+1}^\infty\circ\varphi_{\infty}\circ 2.
\]
In particular, $\varphi_{2^{r+1}}\circ\imath^r_{r+1}\circ\rho_{r}^\infty=\imath^r_{r+1}\circ\varphi_{2^r}\circ\rho_{r}^\infty$. But $\rho_{r}^\infty$ is surjective, so $\varphi_{2^{r+1}}\circ\imath^r_{r+1}=\imath^r_{r+1}\circ\varphi_{2^r}$.

Second, we prove~\ref{lemma_commutator property_2^r}. Observe that 
\[
2[\mathcal{P}_r,\varphi](x)=[\mathcal{P}_r,\varphi](2x)=4[\mathcal{P}_r,\varphi](x)
\]
due to Lemma~\ref{lemma_Pontryagin properties}~\ref{lemma_P properties const square} and Statement~\ref{lemma_commutator property_linear}.
Therefore $2[\mathcal{P}_r,\varphi](x)=0$ and Statement~\ref{lemma_commutator property_2^r} follows.

Third, we prove~\ref{lemma_commutator property_zero on integral}.
Let $x\in H^2(X;\mathbb{Z}/2^{r+1})$. Using the equation $\varphi_{2^r}\circ\rho_{r}^{r+1}=\rho_{r}^{r+1}\circ\varphi_{2^{r+1}}$ together with Lemma~\ref{lemma_Pontryagin properties}~\ref{lemma_P properties reduct n sq} and the fact that $\varphi_{2^{r+1}}$ is a ring homomorphism gives
\[
[\mathcal{P}_r,\varphi](\rho^{r+1}_{r}(x))=\mathcal{P}_r(\varphi_{2^r}(\rho_{r}^{r+1}(x)))-\varphi_{2^{r+1}}(\mathcal{P}_r(\rho_{r}^{r+1}(x)))=\varphi_{2^{r+1}}(x)^2-\varphi_{2^{r+1}}(x^2)=0.
\]

Lastly, we prove~\ref{lemma_commutator of composition} and~\ref{lemma_commutator of inverse}. Noting the order reversal in Equation~\eqref{eq_composition}, Statement~\ref{lemma_commutator of composition} may be checked directly. Statement~\ref{lemma_commutator of inverse} then follows by using~\ref{lemma_commutator of composition} to get
\[
0=[\mathcal{P}_r,id]=[\mathcal{P}_r,\varphi\circ\varphi^{-1}]=\varphi^{-1}_{2^{r+1}}\circ[\mathcal{P}_r,\varphi]+[\mathcal{P}_r,\varphi^{-1}]\circ\varphi_{2^{r}}.
\]
Rearranging this equation and using the fact from Statement~\ref{lemma_commutator property_2^r} that each commutator has order $2$ yields the result.
\end{proof}

We now come to the main result of this section, which computes an upper bound on $h(X)$.

\begin{theorem}\label{prop_h(X)<2}
Let $X$ be a complex in $\mathscr{C}_{n,\tor}$. If $\tor$ is odd, then $h(X)=1$. Otherwise, $h(X)\leq 2$.
\end{theorem}

\begin{proof}
Suppose given proper isomorphisms $\varphi\colon X\dashrightarrow X'$ and $\varphi'\colon X\dashrightarrow X''$ such that $[\mathcal{P}_r,\varphi]\neq0$ and $[\mathcal{P}_r,\varphi']\neq0$. Putting $\theta=\varphi'\circ\varphi^{-1}\colon X'\dashrightarrow X''$ yields a proper isomorphism and Lemma~\ref{lemma_properties of commutator} implies that
\[
[\mathcal{P}_r,\theta]=\varphi^{-1}_{2^{r+1}}\circ[\mathcal{P}_r,\varphi']+\varphi^{-1}_{2^{r+1}}\circ[\mathcal{P}_r,\varphi]\circ\theta_{2^{r}}.
\]

Now, if $\{u_1,\ldots,u_n;w;v\}$ is a cellular basis for $X''$, then Lemma~\ref{lemma_properties of commutator} implies that $[\mathcal{P}_r,\theta]$ and $[\mathcal{P}_r,\varphi']$ vanish on each $u_i$. Thus $[\mathcal{P}_r,\varphi]$ vanishes on each $\theta(u_i)$. Since $[\mathcal{P}_r,\varphi]$ and $[\mathcal{P}_r,\varphi']$ are nontrivial, $[\mathcal{P}_r,\varphi'](w)$ and $[\mathcal{P}_r,\varphi](\theta_{2^r}(w))$ are nontrivial. But the commutators have order $2$, so $[\mathcal{P}_r,\varphi'](w)=[\mathcal{P}_r,\varphi](\theta_{2^r}(w))$ is the unique element of order $2$ in $H^4(X;\mathbb{Z}/2^{r+1})$. Hence $\varphi_{2^{r+1}}([\mathcal{P}_r,\theta](w))=0$, implying that $[\mathcal{P}_r,\theta]=0$. It follows now from Lemma~\ref{lemma_criteria} that $\theta$ is induced by a homotopy equivalence $X'\simeq X''$. We conclude that $h(X)\leq2$. 

Finally, if $\tor$ is odd, then $\mathcal{P}_r$ acts trivially on $H^*(X)$ and $H^*(X')$. Therefore $[\mathcal{P}_r,\varphi]$ is always zero and we get $h(X)=1$.
\end{proof}

\begin{example}\label{ex_propisonothomo}
There are properly isomorphic $4$-dimensional CW-complexes in $\mathscr{C}_{0,2}$ with different homotopy types. 
Consider $\Sigma\R P^3$ and $P^3(2)\vee S^4$. Both are suspensions, so their cohomology rings are trivial and it is easy to construct a proper isomorphism between them. However, $\mathcal{P}_2$ acts trivially on $H^2(P^3(2)\vee S^4;\Z/2)$ but non-trivially on $H^2(\Sigma\R P^3;\Z/2)$, so the two spaces are not homotopy equivalent. In particular, $h(\Sigma\mathbb{R}P^3)=2$, so the inequality in Theorem~\ref{prop_h(X)<2} has a sharp upper bound.
\end{example}

\section{Cohomological rigidity of toric orbifolds}\label{sect_toric orbifold}

In this section we apply the theory of proper isomorphisms developed in the previous section to investigate the cohomological rigidity of $4$-dimensional toric orbifolds. To begin with, we recall the combinatorial description of toric orbifolds.
Let $P$ be a $d$-dimensional simple convex polytope with facets $F_1,\ldots,F_m$.
Being simple means that each face $E$ of codimension-$k$ in $P$ can be written as the intersection of $k$ facets
\[
E=F_{i_1}\cap\cdots\cap F_{i_k},
\]
where $F_{i_1},\ldots,F_{i_k}$ are unique up to permutation.
A \emph{characteristic function} on $P$ is a map
\[
\lambda\colon\{F_1,\ldots,F_m\}\to\Z^d
\]
such that
\begin{itemize}
\item
for each $1\leq i\leq m$, the \emph{characteristic vector} $\lambda_i=\lambda(F_i)$ is primitive
\item
if $F_{i_1},\ldots,F_{i_k}$ have non-empty intersection, then $\lambda_{i_1},\ldots,\lambda_{i_k}$ are linearly independent.
\end{itemize}

Given a simple polytope $P$ and a characteristic function $\lambda$, the associated \emph{toric orbifold} is the quotient space
\[
X=P\times T^d/\sim,
\]
where $\sim$ is the relation defined by letting $(p,g)\sim(q,h)$ if and only if both the following conditions are met
\begin{enumerate}
\item
$p=q$;
\item
$gh^{-1}\in\{\exp(\sum^k_{j=1}t_j\lambda_{i_j})\mid t_1,\ldots,t_k\in\R\}$ when $p$ is an interior point in $F_{i_1}\cap\cdots\cap F_{i_k}$.
\end{enumerate}
Denote by $[p,g]$ the image of $(p,g)\in P\times T^d$ in $X$ and let $T^d$ act on $X$ as
\[
t\cdot[p,g]=[p,t\cdot g]
\quad
\text{for }t\in T^d.
\]
The orbit space of this action is $P$ and the orbit projection $\pi\colon X\to P$ is $\pi([p,g])=p$.

In this paper we focus on the case where $d=2$. Then $P$ is a two dimensional convex polygon with $m$ edges. For convenience we write $m=n+2$. Label vertices and edges of $P$ by $v_1,\ldots,v_{n+2}$ and $F_1,\ldots,F_{n+2}$ as shown in Figure~\ref{fig_polygon}.
Throughout the paper, all indices are taken modulo $n+2$, that is $F_{n+3}=F_1$ and $F_{0}=F_{n+2}$.

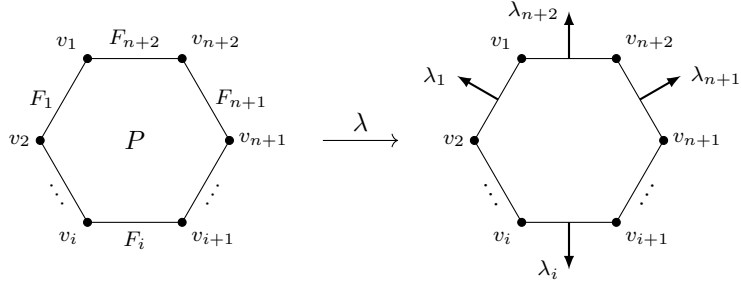
\begin{figure}
\begin{tikzpicture}
\node[opacity=0, regular polygon, regular polygon sides=6, draw, minimum size = 2.5cm](m) at (0,0) {};
\coordinate (1) at (m.corner 1); \coordinate (2) at (m.corner 2); \coordinate (3) at (m.corner 3); 
\coordinate (4) at (m.corner 4); 
\coordinate (5) at (m.corner 5); 
\coordinate (6) at (m.corner 6); 

\draw (1)--(2)--(3)--(4)--(5)--(6)--cycle;

\foreach \a in {1,...,6} {
\draw[fill] (\a) circle (1.5pt); 
}

 \node at (0,0) {$P$};

 \node[above] at (m.side 1) {\footnotesize$F_{n+2}$};
 \node[left] at (m.side 2) {\footnotesize$F_{1}$};
 \node[right] at (m.side 6) {\footnotesize$F_{n+1}$};
 \node[below] at (m.side 4) {\footnotesize$F_{i}$};
\node[left, rotate=30] at (m.side 3) {\footnotesize$\vdots$};
 \node[right, rotate=-30] at (m.side 5) {\footnotesize$\vdots$};

 \node[above right] at (m.corner 1){\footnotesize$v_{n+2}$};
 \node[above left] at (m.corner 2){\footnotesize$v_{1}$};
 \node[left] at (m.corner 3){\footnotesize$v_{2}$};
 \node[below left] at (m.corner 4){\footnotesize$v_{i}$};
 \node[below right] at (m.corner 5){\footnotesize$v_{i+1}$};
 \node[right] at (m.corner 6){\footnotesize$v_{n+1}$};

\draw[->] (2.5,0)--(3.5,0);
\node[above] at (3,0) {$\lambda$};

\begin{scope}[xshift=150]
\node[opacity=0, regular polygon, regular polygon sides=6, draw, minimum size = 2.5cm](n) at (0.5,0) {};
\coordinate (1) at (n.corner 1); 
\coordinate (2) at (n.corner 2); 
\coordinate (3) at (n.corner 3); 
\coordinate (4) at (n.corner 4); 
\coordinate (5) at (n.corner 5); 
\coordinate (6) at (n.corner 6); 

 \draw (1)--(2)--(3)--(4)--(5)--(6)--cycle;

\foreach \a in {1,...,6} {
\draw[fill] (\a) circle (1.5pt); 
}
\end{scope}

\node[left, rotate=30] at (n.side 3) {\small$\vdots$};
 \node[right, rotate=-30] at (n.side 5) {\small$\vdots$};

 \node[above right] at (n.corner 1){\footnotesize$v_{n+2}$};
 \node[above left] at (n.corner 2){\footnotesize$v_{1}$};
 \node[left] at (n.corner 3){\footnotesize$v_{2}$};
 \node[below left] at (n.corner 4){\footnotesize$v_{i}$};
 \node[below right] at (n.corner 5){\footnotesize$v_{i+1}$};
 \node[right] at (n.corner 6){\footnotesize$v_{n+1}$};

\draw [-latex, thick] (n.side 1) -- ($(n.side 1)!1!90:(n.corner 1)$) node [left] at ($(n.side 1)!1!90:(n.corner 1)$){\footnotesize$\lambda_{n+2}$};
\draw [-latex, thick] (n.side 2) -- ($(n.side 2)!1!90:(n.corner 2)$) node [left] at ($(n.side 2)!1!90:(n.corner 2)$){\footnotesize$\lambda_1$};
\draw [-latex, thick] (n.side 4) -- ($(n.side 4)!1!90:(n.corner 4)$) node [left] at ($(n.side 4)!1!90:(n.corner 4)$){\footnotesize$\lambda_i$};
\draw [-latex, thick] (n.side 6) -- ($(n.side 6)!1!90:(n.corner 6)$) node [right] at ($(n.side 6)!1!90:(n.corner 6)$){\footnotesize$\lambda_{n+1}$};
\end{tikzpicture}
\caption{Labels of edges and vertices in polygon $P$}\label{fig_polygon}
\end{figure}

Thanks to the work of~\cite[Lemma 3.1]{Fis}, \cite[Theorem 2.5.5]{Jor} and~\cite[Corollary 5.1]{KMZ}, the cohomology groups of a $4$-dimensional toric orbifold are given as follows:
\[
\begin{tabular}{C{2.4cm}|C{1cm}|C{1cm}|C{1cm}|C{1cm}|C{1cm}|C{1cm}}
$i$		&$0$	&$1$	&$2$	&$3$	&$4$	&$\geq5$\\
\hline
$H^i(X;\Z)$	&$\Z$	&$0$	&$\Z^n$	&$\Z/\tor$	&$\Z$	&$0$
\end{tabular}
\]
where $\tor$ is the greatest common divisor $\gcd\{\det(\lambda_i,\lambda_j)\mid 1\leq i<j\leq n+2\}$.

If $X$ is a $4$-dimensional toric orbifold, $n$ is the rank of $H^2(X;\Z)$, and $\tor$ is the order of $H^3(X;\Z)$, then~$X$ is contained in $\mathscr{C}_{n,\tor}$, since it is simply-connected and by~\cite[Theorem~4H.3]{Hat} admits a minimal cell structure of the form
\begin{equation}\label{eq_toric orb min cell structure}
X\simeq\left(\bigvee^n_{i=1}S^2\vee P^3(\tor)\right)\cup e^4.
\end{equation}
To understand the cohomological rigidity of $4$-dimensional toric orbifolds, Lemma~\ref{lemma_criteria} indicates that we need to study how the Pontryagin square
\[
\mathcal{P}_r\colon H^2(-;\Z/2^r)\to H^4(-;\Z/2^{r+1}),
\quad
\text{where }\nu_2(\tor)=2^r,
\]
acts on the cohomology groups of $X$.
We will show that this action is non-trivial by using the $q$-CW structure of $4$-dimensional toric orbifolds introduced in~\cite{BSS}.

Draw a line segment $\ell$ intersecting $F_{n+1}$ and $F_{n+2}$. The neighbourhood of $v_{n+2}$ bounded by~$\ell$ is the join $\ell*\{v_{n+2}\}$. Letting $P'=\overline{P\setminus \ell*\{v_{n+2}\}}$ be the closure of its complement there is a decomposition
\[
X\cong\pi^{-1}(\ell*\{v_{n+2}\})\cup_{\pi^{-1}(\ell)}\pi^{-1}(P').
\]
See Figure~\ref{fig_q CW}.

\begin{figure}
\begin{tikzpicture}[scale=0.6]

\draw[fill=yellow, thick](30:2)--(90:2)--(150:2)--(210:2)--(270:2)--(330:2)--cycle;
\draw[fill=cyan, thick] (60:1.7)--(90:2)--(120:1.7)--cycle;

\foreach \vangle in {30, 90,150,210,270,330}
\draw[fill] (\vangle:2) circle (2pt); 

\node[right] at (30:2) {\footnotesize$v_{n+1}$};
\node[above] at (90:2) {\footnotesize$v_{n+2}$};
\node[left] at (150:2) {\footnotesize$v_1$};
\node[left] at (210:2) {\footnotesize$v_2$};
\node[below] at (270:2) {\footnotesize$\cdots$};
\node[right] at (330:2) {\footnotesize$v_{n}$};

\node[right] at (60:2) {\footnotesize$F_{n+1}$};
\node[left] at (120:2) {\footnotesize$F_{n+2}$};
\node[left] at (180:2) {\footnotesize$F_1$};
\node[right] at (0:2) {\footnotesize$F_{n}$};
\draw[fill] (60:1.7) circle (1pt); 
\draw[fill] (120:1.7) circle (1pt); 
\draw[red, very thick] (60:1.7)--(120:1.7);
\node at (90:1) {$\ell$};
\node at (0:4.2) {$=$};

\begin{scope}[xshift=230]
\draw[fill=yellow, thick](30:2)--(60:1.7)--(120:1.7)--(150:2)--(210:2)--(270:2)--(330:2)--cycle;
\draw[red, very thick] (60:1.7)--(120:1.7);
\node at (90:1) {$\ell$};
\foreach \vangle in {30,150,210,270,330}
\draw[fill] (\vangle:2) circle (2pt); 

\node[right] at (30:2) {\footnotesize$v_{n+1}$};
\node[left] at (150:2) {\footnotesize$v_1$};
\node[left] at (210:2) {\footnotesize$v_2$};
\node[below] at (270:2) {\footnotesize$\cdots$};
\node[right] at (330:2) {\footnotesize$v_{n}$};

\node[left] at (180:2) {\footnotesize$F_1$};
\node[right] at (0:2) {\footnotesize$F_{n}$};

\node at (0:4.2) {\Large$\cup$};

\end{scope}

\begin{scope}[xshift=410]

\draw[fill=cyan, thick] (60:1.7)--(90:2)--(120:1.7)--cycle;

\draw[fill] (90:2) circle (2pt);
\node[above] at (90:2) {\footnotesize$v_{n+2}$};

\draw[red, very thick] (60:1.7)--(120:1.7);
\node at (90:1) {$\ell$};

\end{scope}
\end{tikzpicture}
\caption{$q$-CW structure of $X(P,\lambda)$}\label{fig_q CW}
\end{figure}
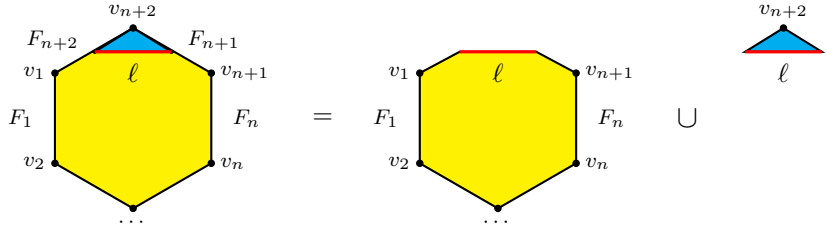

Note that $\pi^{-1}(\ell*\{v_{n+2}\})$ is homeomorphic to the cone of $\pi^{-1}(\ell)$, so $X$ is the mapping cone of the inclusion $\pi^{-1}(\ell)\hookrightarrow\pi^{-1}(P')$.
Following the arguments in~\cite{BSS,FSS1,So},
one can show that~$\pi^{-1}(\ell)$ is homeomorphic to a lens space $L=L(b,a)$, where $b=|\det(\lambda_{n+1},\lambda_{n+2})|$, and the deformation retraction $P'\overset{\simeq}{\longrightarrow}\bigcup^{n}_{i=1}F_i$ induces a homotopy equivalence $\pi^{-1}(P')\simeq\bigvee^{n}_{i=1}S^2_i$. Here~$S^2_i$ is a pointed~$2$-sphere obtained from $\pi^{-1}(F_i)$.
Let $f\colon L\to\bigvee^{n}_{i=1}S^2_i$ be the composite
\[
f\colon L\cong\pi^{-1}(\ell)\hookrightarrow\pi^{-1}(P')\simeq\bigvee^{n}_{i=1}S^2_i.
\]
There is a homotopy cofibration sequence
\begin{equation}\label{eqn_toric orb cofib seq}
L\overset{f}{\longrightarrow}\bigvee^{n}_{i=1}S^2_i\longrightarrow X.
\end{equation}
Note that the wedge $\bigvee^n_{i=1}S^2_i$ in~\eqref{eqn_toric orb cofib seq} should not be confused with the wedge $\bigvee^n_{i=1}S^2$ in~\eqref{eq_toric orb min cell structure}. The cohomology classes corresponding to the latter wedge form a basis for $H^2(X;\Z)$, whereas those of the former do not.

\begin{lemma}\label{lemma_Pontryagin of torsion element in X}
Let $X$ be a $4$-dimensional toric orbifold where $H^3(X;\Z)\cong\Z/\tor$. If $\nu_2(\tor)=2^r$ and~$r\geq 1$, then there exists a generator $w\in H^2(X;\Z/2^r)$ such that
\[
\beta_{r}(w)\neq 0
\quad\text{and}\quad
\mathcal{P}_r(w)\equiv 2^r\pmod{2^{r+1}}.
\]
\end{lemma}

\begin{proof}
By~\cite[Lemma~5.2]{FSS0} there exists a vertex $v_i\in P$ such that $\nu_2(\det(\lambda_i,\lambda_{i+1}))=\nu_2(\tor)$.
Relabelling vertices if necessary, we may assume that this is $v_{n+2}$.
Applying the above construction, we obtain a homotopy cofibration sequence~\eqref{eqn_toric orb cofib seq} where, by our choice of $v_{n+2}$, the lens space $L=L(b,a)$ satisfies $\nu_2(b)=\nu_2(\tor)$. 
The homotopy cofibration sequence induces an exact sequence
\[
H^2(L;\Z/2^r)\overset{\delta}{\to}H^3(X;\Z/2^r)\to 0 \to H^3(L;\Z/2^r)\overset{\delta}{\to}H^4(X;\Z/2^r)\to 0.
\]
Since $H^2(L;\Z/2^r)\cong H^3(X;\Z/2^r)\cong\Z/2^r$, both connecting maps $\delta$ are isomorphisms.

Let $x_r$ be a generator of $H^1(L;\Z/2^r)$. Then $w=\delta(x_r)$ is a generator of $H^2(X;\Z/2^r)$.
By Lemma~\ref{lemma_Postnikov properties}~\ref{lemma_Postnikov properties_commute with connecting maps} there is a commutative diagram
\[
\xymatrix{
H^1(L;\Z/2^r)\ar[r]^-{\delta}\ar[d]^-{\mathcal{P}'_r}	&H^2(X;\Z/2^{r})\ar[d]^-{\mathcal{P}_r}\\
H^3(L;\Z/2^{r+1})\ar[r]^-{\delta}					&H^4(X;\Z/2^{r+1}).
}
\]
It follows from Proposition~\ref{lemma_Postnikov sq in L} that $\mathcal{P}'_r(x_r)\equiv 2^r\pmod{2^{r+1}}$.
Since the bottom $\delta$ is an isomorphism, we have
\[
\mathcal{P}_r(w)\equiv2^r\pmod{2^{r+1}}.
\]

It remains to show that $w$ has a non-trivial Bockstein image. This follows from the equation $\beta_r(w)=\beta_r(\delta(x_r))=\delta(\beta_r(x_r))$, the fact that $\beta_r(x_r)\neq0$, and that fact that $\delta$ is an isomorphism.
\end{proof}

Next, we apply Lemma~\ref{lemma_criteria} to prove the cohomological rigidity for a special case of $4$-dimensional toric orbifolds.

\begin{theorem}\label{thm_cohmlg rigidity holds for r=n=1}
Let $X$ and $X'$ be $4$-dimensional toric orbifolds such that
\[
H^2(X;\Z)\cong H^2(X';\Z)\cong\Z^n
\quad\text{and}\quad
H^3(X;\Z)\cong H^3(X';\Z)\cong\Z/\tor.
\]
Suppose $\nu_2(\tor)=2$ and $n=1$.
Then $X$ and $X'$ are homotopy equivalent if and only if they are properly isomorphic.
\end{theorem}

\begin{proof}
Let $\varphi\colon X'\dashrightarrow X$ be a proper isomorphism.
The theorem holds if $[\mathcal{P}_1,\varphi]=0$ by Lemma~\ref{lemma_criteria}. Therefore we assume $[\mathcal{P}_1,\varphi]\neq0$.

Let $\{u;w;v\}$ be a cellular basis for $X$. By Lemma~\ref{lemma_basis define skeleton} $X'$ has a cellular structure such that
\[
\{u'=\varphi_{\infty}(u);w'=\varphi_{\tor}(w);v'=\varphi_{\infty}(v)\}
\]
is its cellular basis. Since $Im([\mathcal{P}_1,\varphi])=\{0,2\}\subset\Z/4$ and $[\mathcal{P}_1,\varphi](u_i)=0$ by Lemma~\ref{lemma_properties of commutator}~\ref{lemma_commutator property_2^r} and~\ref{lemma_commutator property_zero on integral}, we have
\[
\mathcal{P}_1(w)\equiv\mathcal{P}_1(w')+2\pmod{4}.
\]
Without loss of generality we assume that $\mathcal{P}_1(w)=2$ and $\mathcal{P}_1(w')=0$. Otherwise, interchange the roles of $X$ and $X'$.
By Lemma~\ref{lemma_Pontryagin of torsion element in X}, there exists a generator $\tilde{w}\in H^2(X';\Z/\tor)$ such that
\[
\beta_{1}(\tilde{w})\neq0
\quad\text{and}\quad
\mathcal{P}_1(\tilde{w})\equiv2\pmod{4}.
\]
It must be the case that $\tilde{w}\equiv w'+u'\pmod{2}$. Therefore $\mathcal{P}_1(w)\equiv\mathcal{P}_1(w'+u')\equiv2\pmod{4}$.

Define a new proper isomorphism $\tilde{\varphi}'\colon X'\dashrightarrow X$ by setting
\[
\tilde{\varphi}_{t}(x)=\begin{cases}
w'+\frac{\tor}{2}u'	&\text{if $t=\tor$ and $x=w$}\\
\varphi_{t}(x)			&\text{else}.
\end{cases}
\]
To prove its well-definedness, we need to show that the $\tilde{\varphi}_{t}$'s are ring isomorphisms and make Diagram~\eqref{dgrm_proper isom commute rho} commute.

First, we show that the $\tilde{\varphi}_t$'s are ring isomorphisms. It suffices to check
\begin{equation}\label{eq_check tilde varphi ring isom 1}
\tilde{\varphi}_{\tor}(u)\cup\tilde{\varphi}_{\tor}(w)=\tilde{\varphi}_{\tor}(u\cup w)
\quad\Leftrightarrow\quad
u'\cup(w'+\frac{\tor}{2}u')\equiv u'\cup w'\pmod{\tor}
\end{equation}
and
\begin{equation}\label{eq_check tilde varphi ring isom 2}
\tilde{\varphi}_{\tor}(w)\cup\tilde{\varphi}_{\tor}(w)=\tilde{\varphi}_{\tor}(w\cup w)
\quad\Leftrightarrow\quad
(w'+\frac{\tor}{2}u')\cup(w'+\frac{\tor}{2}u')\equiv w'\cup w'\pmod{\tor}.
\end{equation}
According to~\cite[Theorem 1.2]{FSS1} or~\cite[Corollary 1.4]{So}, we have $u'\cup u'\equiv0\pmod{2}$, implying the right identities in~\eqref{eq_check tilde varphi ring isom 1} and~\eqref{eq_check tilde varphi ring isom 2}:
\[
u'\cup (w'+\frac{\tor}{2}u')=u'\cup w'+\frac{\tor}{2}u'\cup u'\equiv u'\cup w'\pmod{\tor}
\]
and
\[
(w'+\frac{\tor}{2}u')\cup(w'+\frac{\tor}{2}u')=w'\cup w'+\tor w'\cup u'+\frac{\tor^2}{4}u'\cup u'\equiv w'\cup w'\pmod{\tor}.
\]
Therefore the $\tilde{\varphi}_{t}$'s are ring isomorphisms.

Next we show that the $\tilde{\varphi}_{t}$'s make Diagram~\eqref{dgrm_proper isom commute rho} commute.
It suffices to prove that
\begin{equation}\label{eq_check tilde varphi reduction}
\tilde{\varphi}_{\tor}\circ\rho(x)=\rho\circ\tilde{\varphi}_{2\tor}(x)
\end{equation}
where $\rho$ is the reduction $\Z/2\tor\to\Z/\tor$.
Let $\mu$ be the mod-$2\tor$ image of $u\in H^2(X;\Z)$
and let $\imath$ be the inclusion $\Z/\tor\to\Z/2\tor$. Then $\{\mu,\imath(w)\}$ forms a basis for $H^2(X;\Z/2\tor)$.
Equation~\eqref{eq_check tilde varphi reduction} obviously holds for $\mu$. On the other hand, $\rho\circ\imath(w)=2w$ so
\[
\tilde{\varphi}_{\tor}\circ\rho(\imath w)=2(w'+\frac{\tor}{2}u')=2w'=\varphi_{\tor}\circ\rho(\imath(w))=\rho\circ\varphi_{2\tor}(\imath(w))=\rho\circ\tilde{\varphi}_{2\tor}(\imath (w)).
\]
Therefore the $\tilde{\varphi}_{t}$'s make Diagram~\eqref{dgrm_proper isom commute rho} commute and hence $\tilde{\varphi}$ is well-defined.

Now, since $\nu_2(\tor)=2$, we have
$\tilde{\varphi}_{\tor}(w)=w'+\frac{\tor}{2}u'\equiv w'+u'\pmod{2}$. It follows that
\[
[\mathcal{P}_1,\tilde{\varphi}](w)=\mathcal{P}_1\circ\tilde{\varphi}_{\tor}(w)-\tilde{\varphi}_{2\tor}\circ\mathcal{P}_1(w)
\equiv\mathcal{P}_1(w'+u')-\mathcal{P}_1(w)\equiv0\pmod{4}.
\]
By Lemma~\ref{lemma_criteria} there is a homotopy equivalence $X\simeq X'$.
\end{proof}

\begin{remark}\label{remark_extend main proof}
The proof of Theorem~\ref{thm_cohmlg rigidity holds for r=n=1} can be extended to hold also when $n\geq1$ and $\nu_2(\tor)=2$. It requires the fact that $u_j\cup u_k\equiv0\pmod{2}$ for all degree-two cohomology classes in $H^2(X;\Z)$. This will be proved in a forthcoming paper~\cite{FSS2}, currently in preparation.
\end{remark}

Now we have all ingredients to prove the main Theorem.

\begin{proof}[Proof of Theorem~\ref{thm_main}]
Part~\ref{thm_main_part b} is implied by Theorem~\ref{prop_h(X)<2} since $X,X'$ and $X''$ are CW-complexes in the category $\mathscr{C}_{n,\tor}$ by assumption.
Part~\ref{thm_main_part a} follows directly from Theorem~\ref{thm_cohmlg rigidity holds for r=n=1}.
\end{proof}

\bibliographystyle{amsalpha}

\end{document}